\documentclass[12pt,a4wide]{article}
\usepackage{graphics}
\usepackage{graphicx}
\usepackage{caption}
\usepackage{floatrow}
\usepackage{float}
\usepackage{pifont}
\usepackage{subcaption}
\usepackage{epsfig}
\usepackage[mathscr]{euscript}
\usepackage{mathrsfs}
\usepackage{amsmath,amsthm,amsfonts}
\usepackage{amssymb}
\usepackage{enumerate}
\usepackage{hyperref}
\newtheorem{theorem}{Theorem}[section]

\newtheorem{proposition}[theorem]{Proposition}
\newtheorem{prop}[theorem]{Proposition}
\newtheorem{definition}[theorem]{Definition}

\newtheorem{lemma}[theorem]{Lemma}
\newtheorem{corollary}[theorem]{Corollary}
\newtheorem{remark}[theorem]{Remark}

{

\newcommand{\Ho}{\mathfrak{H}}
\newcommand{\ho}{\mathfrak{h}}
\newcommand{\Go}{\mathfrak{G}}
\newcommand{\go}{\mathfrak{g}}

\begin{document}
\title{A generalization of a theorem of Hoffman}

\author{Jack H. Koolen$^{\dag\S}$, Qianqian Yang$^\dag$, Jae Young Yang$^\ddag$\footnote{Corresponding author}~
\\ \\
\small $^\dag$School of Mathematical Sciences\\
\small University of Science and Technology of China \\
\small 96 Jinzhai Road, Hefei, 230026, Anhui, PR China\\
\small $^\S$Wen-Tsun Wu Key Laboratory of CAS\\
\small 96 Jinzhai Road, Hefei, 230026, Anhui, PR China\\
\small $^\ddag$School of Mathematical Sciences\\
\small Anhui University\\
\small 111 Jiulong Road, Hefei, 230039, Anhui, PR China\\
\small {\tt koolen@ustc.edu.cn, xuanxue@mail.ustc.edu.cn, rafle@postech.ac.kr}\vspace{-3pt}
}
\maketitle
\date{}
\vspace{-24pt}
\begin{abstract}
In 1977, Hoffman gave a characterization of graphs with smallest eigenvalue at least $-2$.  In this paper we generalize this result to graphs with smaller smallest eigenvalue. For the proof, we use a combinatorial object named Hoffman graph, introduced by Woo and Neumaier in 1995.
Our result says that for every $\lambda \leq -2$, if a graph with smallest eigenvalue at least $\lambda$ satisfies some local conditions, then it is highly structured. We apply our result to graphs which are cospectral with the Hamming graph $H(3,q)$, the Johnson graph $J(v, 3)$ and the $2$-clique extension of grids, respectively.
\end{abstract}

\textbf{Keywords} : smallest eigenvalue, Hoffman graph, Johnson graph, Hamming graph, $2$-clique extension of grid graph, intersection graph, hypergraph

\textbf{AMS classification} : 05C50, 05C75, 05C62

\section{Introduction}
In 1976, Cameron et al. \cite{Cameron} showed that a connected graph with smallest eigenvalue at least $-2$ is a generalized line graph if the number of vertices is at least 37. For the proof, the classification of the irreducible root lattices was essential. In 1977, Hoffman \cite{Hoff1977} showed the following result. Note that in this paper, we denote by $d(x)$ the degree of $x$, $G_x$ the induced subgraph on the neighbors of $x$, and $\bar{d}(G_x)$ the average degree of $G_x$.
\begin{theorem}\label{Hoffman}
Let $-1-\sqrt{2}<\lambda\le -2$ be a real number. Then, there exists an integer $f(\lambda)>0$, such that if $G$ is a graph satisfying
\begin{enumerate}
\item $d(x)\geq f(\lambda)$ holds for all $x \in V(G)$,
\item $\lambda\le\lambda_{\min}(G)\le-2$,
\end{enumerate}
then $G$ is a generalized line graph and $\lambda_{\min}(G)=-2$.
\end{theorem}
For the proof of the above theorem, Hoffman did not use the classification of the irreducible root lattices. This meant that he had to assume large minimum degree. In this paper, we use ideas from Hoffman \cite{Hoff1977} to show four structural results that generalize Theorem \ref{Hoffman}, where the first two results give sufficient conditions for a graph to be the intersection graph of some linear uniform hypergraph. The notions of uniform hypergraphs and $p$-plexes will be introduced in the subsequent section.

\begin{theorem}\label{intro1}
Let $t$ be a positive integer. Then, there exists a positive integer $K(t)$, such that if a graph $G$ satisfies the following conditions:
\begin{enumerate}
\item $d(x)>K(t)$ holds for all $x \in V(G)$,
\item any $(t^2 +1)$-plex containing $x$ has order at most $\frac{d(x)-K(t)}{t}$ for all $x \in V(G)$,
\item $\lambda_{\min} (G) \geq -t-1$,
\end{enumerate}
then $G$ is the intersection graph of some linear $(t+1)$-uniform hypergraph.
\end{theorem}

\begin{theorem}\label{intro2}
Let $t$ be a positive integer. Then, there exists a positive integer $\kappa(t)$, such that if a graph $G$ satisfies the following conditions:
\begin{enumerate}
\item $d(x)>\kappa(t)$ holds for all $x \in V(G)$,
\item $\bar{d}(G_x) \leq \frac{d(x)-\kappa(t)}{t}$ holds for all $x \in V(G)$,
\item $\lambda_{\min} (G) \geq -t-1$,
\end{enumerate}
then $G$ is the intersection graph of some linear $(t+1)$-uniform hypergraph.
\end{theorem}

Note that for $t=1$, Theorem \ref{intro1} and Theorem \ref{intro2} are weaker than Theorem \ref{Hoffman}. For $t=2$, we have the following two results on graphs with smallest eigenvalue at least $-3$.  In order to state these results, we introduce three Hoffman graphs $\mathfrak{f}_1$, $\mathfrak{f}_2$ and $\mathfrak{f}_3$ as follows:
\begin{center}
\begin{figure}[H]
\includegraphics[scale=1.0]{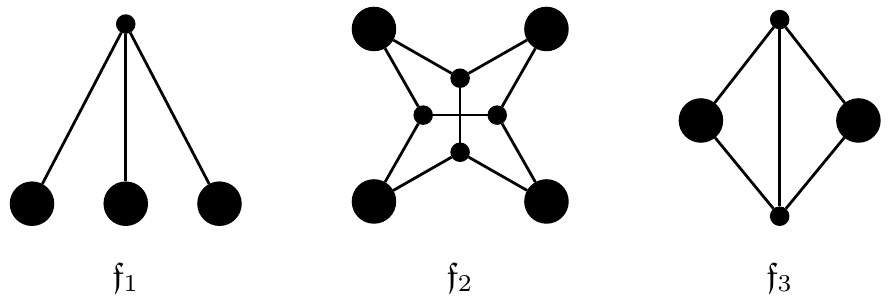}
{\caption{}\label{figure}}
\end{figure}
\end{center}
\vspace{-1cm}
(For the notions of Hoffman graphs and line Hoffman graphs, we refer to the next section.)

\begin{theorem}\label{intro3}
There exists a positive integer $K$, such that if a graph $G$ satisfies the following conditions:
\begin{enumerate}
\item $d(x)>K$ holds for all $x \in V(G)$,
\item any $5$-plex containing $x$ has order at most $d(x)-K$ for all $x \in V(G)$,
\item $\lambda_{\min} (G) \geq -3$,
\end{enumerate}
then $G$ is the slim graph of a $2$-fat $\{\mathfrak{f}_1,\mathfrak{f}_2,\mathfrak{f}_3\}$-line Hoffman graph.
\end{theorem}

\begin{theorem}\label{intro4}
There exists a positive integer $\kappa$, such that if a graph $G$ satisfies the following conditions:
\begin{enumerate}
\item $d(x)>\kappa$ holds for all $x \in V(G)$,
\item $\bar{d}(G_x)\leq d(x)-\kappa$ holds for all $x \in V(G)$,
\item $\lambda_{\min} (G) \geq -3$,
\end{enumerate}
then $G$ is the slim graph of a $2$-fat $\{\mathfrak{f}_1,\mathfrak{f}_2,\mathfrak{f}_3\}$-line Hoffman graph.
\end{theorem}

As applications of Theorems \ref{intro1}--\ref{intro4}, we give the following three results.

\begin{theorem} \label{Hamming} There exists a positive integer $q^\prime$ such that for each integer $q \geq q^\prime$, any graph, that is cospectral with the Hamming graph $H(3,q)$ is the intersection graph of some linear $3$-uniform hypergraph.
\end{theorem}

\begin{theorem} \label{Johnson} There exists a positive integer $v^\prime$ such that for each integer $v \geq v^\prime$, any graph, that is cospectral with the Johnson graph $J(v,3)$ is the intersection graph of some linear $3$-uniform hypergraph.
\end{theorem}

\begin{theorem} \label{grid} There exists a positive integer $t$ such that for each pair integers $(t_1,t_2)$ with $t_1 \geq t_2 \geq t$, any graph, that is cospectral with the $2$-clique extension of the $(t_1 \times t_2)$-grid is the slim graph of a $2$-fat $\{\mathfrak{f}_1,\mathfrak{f}_2,\mathfrak{f}_3\}$-line Hoffman graph.
\end{theorem}

\begin{remark}
\begin{enumerate}
\item Theorem $\ref{Hamming}$ was first shown by Bang et al. $\cite{Bang}$. They also obtained that $q^\prime\leq 36$.
\item Van Dam et al. \cite{VanDam&Haemers} gave two constructions for non-distance-regular graphs with the same spectrum as $J(v,3)$ for all $v$. One method used Godsil-McKay switching and the other method was to construct such graphs as the point graph of a partial linear space. Theorem \ref{Johnson} states that if $v$ is large enough, these graphs have to be the point graphs of partial linear spaces.
\item In \cite{YAK}, we showed that the $2$-clique extension of the  $(t\times t)$-grid is characterized by its spectrum, if $t$ is large enough. In order to show it, Theorem \ref{grid} was of crucial importance.
\item Very recently, we generalized Theorem \ref{Hoffman} to the class of graphs with smallest eigenvalue at least $-3$. We were able to show that, if $G$ has smallest eigenvalue between $-2$ and $-3$, then there exists an integral matrix $N$ such that $2A(G) + 6I = N^TN$ holds, where $A(G)$ is the adjacency matrix of $G$. Theorems 1.4 and 1.5 are non-trivial refinements of this new result.
\end{enumerate}
\end{remark}

To show Theorems \ref{intro1}--\ref{intro4}, we use the theory of Hoffman graphs as introduced by Woo and Neumaier \cite{Woo}. In Section $2$, we give definitions and basic theory of Hoffman graphs. In Section $3$, we define a set of matrices $\mathcal{M}(t)$ and a finite family of $t$-fat Hoffman graph $\Go(t)$, where $t$ is a positive integer. In Corollary \ref{line1}, we use the set of matrices, $\mathcal{M}(t)$, to show that a $t$-fat Hoffman graph with smallest eigenvalue greater than $-t-\sqrt{2}$ is a $t$-fat $\Go(t)$-line Hoffman graph. In Section $4$, we consider associated Hoffman graphs of graphs, (as introduced by Kim et al. \cite{HJJ}) and show some properties for these associated Hoffman graphs to be used later. In Theorem \ref{main1} and Theorem \ref{main2}, we prove that if a graph $G$ satisfies some local conditions and has smallest eigenvalue at least $-t-1$, then its associated Hoffman graph is a $\Go(t)$-line Hoffman graph. This implies that $G$ is highly structured. Also, in Section $5$, we show Theorems \ref{intro1} and \ref{intro3} (resp. Theorems \ref{intro2} and \ref{intro4}) as a consequence of Theorem \ref{main1} (resp. Theorem \ref{main2}). In the last section, we give proofs of Theorems \ref{Hamming}, \ref{Johnson} and \ref{grid}, by using Theorems \ref{intro1}--\ref{intro4}.

\section{Definitions and preliminaries}
\subsection{Graphs and $h$-uniform hypergraphs}
All the graphs considered in this paper are finite, undirected and simple. For a given graph $G=(V(G),E(G))$, the \emph{adjacency matrix} $A(G)$ of $G$ is the $(0,1)$-matrix with rows and columns indexed by the vertices of $G$ such that the $uv$-entry of $A(G)$ is equal to $1$ if and only if $u$ and $v$ are adjacent in $G$. The \emph{eigenvalues} of $G$ are the eigenvalues of $A(G)$. The \emph{spectrum} of $G$ is the multiset $$\big\{\lambda_0^{m_0},\lambda_1^{m_1},\ldots,\lambda_t^{m_t}\big\},$$ where  $\lambda_0,\lambda_1,\ldots,\lambda_t$ are the distinct eigenvalues of $G$ and $m_i$ is the multiplicity of $\lambda_i$ ($i=0,1,\ldots,t$). Two graphs are called \emph{cospectral} if they have the same spectrum.

A graph $G$ is called \emph{walk-regular} if the number of closed walks of length $r$ starting at a given vertex $x$ is independent of the choice of $x$ for each $r$. Since this number equals $A^r_{xx}$, it is the same as saying that $A^r$ has constant diagonal for all $r$, where $A=A(G)$ is the adjacency matrix of $G$. It is easy to see that a walk-regular graph is always regular.

For the following discussion, we follow Van Dam \cite{VanDam}. For convenience, in this paper we always denote by $I$ and $J$ the identity matrix and all-one matrix respectively, and we point out their orders by giving subscripts if necessary.

Let $G$ be a connected $k$-regular graph with $n$ vertices and adjacency matrix $A$. Suppose that $G$ has exactly four distinct eigenvalues $\lambda_0=k,\lambda_1,\lambda_2,\lambda_3$, then $(A-\lambda_1I)(A-\lambda_2I)(A-\lambda_3I)=\frac{1}{n}(k-\lambda_1)(k-\lambda_2)(k-\lambda_3)J$. Since $A^2,A,I$ and $J$ all have constant diagonal, we see that $A^r$ has constant diagonal for every $r$. This implies that $G$ is walk-regular. In particular, $A^3_{xx}$, which counts twice the number of triangles through $x$, does not depend on the vertex $x$ and equals  $(\lambda_1+\lambda_2+\lambda_3)k+\lambda_1\lambda_2\lambda_3+\frac{1}{n}(k-\lambda_1)(k-\lambda_2)(k-\lambda_3)$. Hence, for any vertex $x$ in $G$, the local graph $G_x$ of $G$ at $x$ has average degree
\begin{equation}\label{averagevalency}
\bar{d}(G_x)=\lambda_1+\lambda_2+\lambda_3+\frac{1}{k}\lambda_1\lambda_2\lambda_3+\frac{1}{nk}(k-\lambda_1)(k-\lambda_2)(k-\lambda_3).
\end{equation}

In $\cite{Plex}$, a generalization of a clique is introduced as follows:

\begin{definition}\label{plex}
Let $p$ be a positive integer. A $p$-\emph{plex} is an induced subgraph in which each vertex is adjacent to all but at most $p-1$ of the other vertices.
\end{definition}

Note that a clique is exactly the same as a $1$-plex.

A \emph{hypergraph} $H$ is a pair $(X,E)$, where the elements of $E$ are non-empty subsets (of any cardinality) of the finite set $X$. The set $X$ is called the vertex set of $H$ and $E$ is called the (hyper)edge set of $H$. For $h\geq2$, a hypergraph $H$ is said to be \emph{$h$-uniform} if every edge in $H$ has cardinality $h$. The hypergraph $H$ is \emph{linear} if every pair of distinct vertices of $H$ is contained in at most one edge of $H$.
\begin{definition}Suppose $H=(X,E)$ is a hypergraph. The \emph{intersection graph} of $H$ is the graph with vertex set $E$ and edge set consisting of all unordered pairs $\{e,e^\prime\}$ of distinct elements of $E$ such that $e\cap e^\prime\neq\emptyset$.
\end{definition}

Clearly, a $2$-uniform hypergraph is a graph and its intersection graph is the usual line graph.

\subsection{Matrices and interlacing}
Suppose $M_1$ is a real symmetric $n\times n$ matrix, and b$M_2$ is a real symmetric $m\times m$ matrix ($m\leq n$). Let $\theta_1(M_1)\ge\theta_2(M_1)\ge\cdots\ge\theta_n(M_1)$ and $\theta_1(M_2)\ge\theta_2(M_2)\ge\cdots\ge\theta_m(M_2)$ denote their eigenvalues in nonincreasing order. We say that the eigenvalues of $M_2$ \emph{interlace} the eigenvalues of $M_1$ if for $i=1,\ldots,m,$
\begin{equation*}
\theta_{n-m+i}(M_1)\leq \theta_{i}(M_2)\leq\theta_i(M_1).
\end{equation*}

By \cite[Theorem 9.1.1]{Ch&G} and \cite[Lemma 9.6.1]{Ch&G}, we have the following results on interlacing:

\begin{lemma}\label{matrixinterlace}
\begin{enumerate}
\item Suppose $M$ is a real symmetric matrix and $M^\prime$ is any principal submatrix of $M$. Then the eigenvalues of $M^\prime$ interlace the eigenvalues of $M$.
\item Suppose $G$ is a graph and let $\pi$ be a partition of $V(G)$ with cells $V_1,\ldots,V_r$. Define the \emph{quotient matrix} of $A(G)$ relative to $\pi$ to be the $r\times r$ matrix $B$ such that the $ij$-entry of $B$ is the average number of neighbors in $V_j$ of vertices in $V_i$. Then the eigenvalues of $B$ interlace the eigenvalues of $A(G)$.
\end{enumerate}
\end{lemma}

The following result is about an upper bound on the order of a $(p+1)$-plex for a regular graph. It is a generalization of the Hoffman bound on independent sets for regular graphs (see, \cite[Lemma $9.6.2$]{Ch&G}).

\begin{theorem}\label{plexorder}
Let $G$ be a $k$-regular graph with $n$ vertices and let $\theta_1=k \geq \theta_2 \geq \dots \geq \theta_n$ be the eigenvalues of $G$. Then the order of a $(p+1)$-plex in $G$ is at most $\frac{n(p+1 +\theta_2)}{n-k+\theta_2}$.
\end{theorem}

\begin{proof}
Let $P$ be a $(p+1)$-plex of $G$ with $|V(P)|=m$. Consider the partition $\pi=\big\{V(P),V(G)-V(P)\big\}$ of $V(G)$. The quotient matrix $B$ of A(G) relative to $\pi$ is
$$B= \begin{pmatrix}
\alpha & k-\alpha \\
\beta & k-\beta
\end{pmatrix},$$
where $\alpha$ and $\beta$ are real numbers satisfying $\alpha\ge m-(p+1)$ and $\beta= \frac{m(k-\alpha)}{n-m} $. Then $B$ has eigenvalues $k$ and $\alpha-\beta$, as tr$~B=k+\alpha-\beta$. By Lemma \ref{matrixinterlace} (ii), we have $m-(p+1)-\frac{m(k-(m-(p+1)))}{n-m}\leq\alpha-\beta\leq \theta_2$. This completes the proof.
\end{proof}

\subsection{Hoffman graphs and special matrices}
In this subsection we introduce Hoffman graphs and their special matrices.
\begin{definition}\label{hoffmangraph}
A \emph{Hoffman graph} $\ho$ is a pair $(H, \ell)$, where $H=(V,E)$ is a graph and $\ell:V\rightarrow\{{\bf fat},{\bf slim}\}$ is a labeling map on $V$, such that no two vertices with label {\bf fat} are adjacent and every vertex with label {\bf fat} has at least one neighbor with label {\bf slim}.
\end{definition}

The vertices with label {\bf fat} are called \emph{fat vertices}, and the vertices with label {\bf slim} are called \emph{slim vertices}. We denote the set of slim vertices by $V_{{\rm slim}}(\ho)$ and the set of fat vertices by $V_ {\rm fat}(\ho)$. For a vertex $x$ of $\mathfrak{h}$, the set of slim (resp. fat) neighbors of $x$ in $\mathfrak{h}$ is denoted by $N_{\mathfrak{h}}^{{\rm slim}}(x)$ (resp. $N_{\mathfrak{h}}^{{\rm fat}}(x)$).
A Hoffman graph is called \emph{$t$-fat} if every slim vertex has at least $t$ fat neighbors, where $t$ is a positive integer. A \emph{fat} Hoffman graph is a $1$-fat Hoffman graph.

The \emph{slim graph} of $\ho$ is the subgraph of $H$ induced by $V_{\rm slim}(\ho)$. Note that we consider the slim graph of a Hoffman graph as an ordinary graph.

\begin{definition}Suppose $\ho= (H, \ell)$ is a Hoffman graph.
\begin{enumerate}

\item The Hoffman graph $\ho_1 = (H_1, \ell_1)$ is called an \emph{induced Hoffman subgraph} of $\ho$, if $H_1$ is an induced subgraph of $H$ and $\ell_1(x) = \ell(x)$ for all vertices $x$ of $H_1$.

\item Let $W$ be a subset of $V_{\rm slim}(\ho)$. The induced Hoffman subgraph of $\ho$ generated by $W$, denoted by $\langle W\rangle_{\ho}$, is the Hoffman subgraph of $\ho$ induced by $W \cup \{f \in V_{\rm fat}(\ho)$ $|$ $f \sim w$ for some $w \in W \}$.
\end{enumerate}
\end{definition}

\begin{definition}Two Hoffman graphs $\ho= (H, \ell)$ and $\ho^\prime=(H^\prime, \ell^\prime)$ are \emph{isomorphic} if there exists an isomorphism from $H$ to $H^\prime$ which preserves the labeling.
\end{definition}

For a Hoffman graph $\ho= (H, \ell)$, there exists a matrix $C$ such that the adjacency matrix $A$ of $H$ satisfies
\begin{eqnarray*}
A=\left(
\begin{array}{cc}
A_{\rm slim}  & C\\
C^{T}  & O
\end{array}
\right),
\end{eqnarray*}
where $A_{\rm slim}$ is the adjacency matrix of the slim graph of $\mathfrak{h}$. The
real symmetric matrix $S(\ho):=A_{\rm slim}-CC^T$ is called the \emph{special matrix} of $\ho$. The
\emph{eigenvalues} of $\ho$ are the eigenvalues of $S(\ho)$, and the smallest eigenvalue of $\ho$ is denoted by  $\lambda_{\min}(\ho)$.

Note that Hoffman graphs are not determined by their special matrices. Let $\ho_1$ and $\ho_2$ be the two Hoffman graphs shown in Figure \ref{fsamespecialmatrix}. They have both $\begin{pmatrix}
-2 & -1 \\
-1 & -2
\end{pmatrix}$ as their special matrix, but they are not isomorphic.
\begin{center}
\begin{figure}[H]
\includegraphics[scale=1.0]{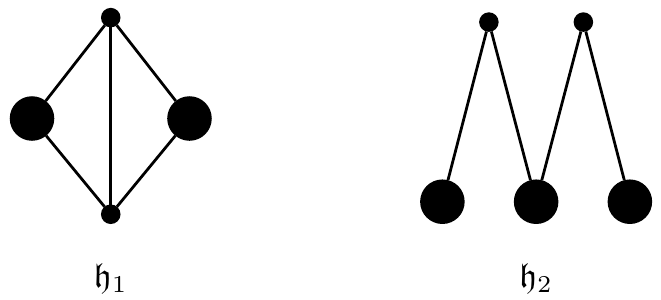}
{\caption{}\label{fsamespecialmatrix}}
\end{figure}
\end{center}
\vspace{-1cm}
\begin{lemma}\label{special} Let $S$ be the special matrix of a Hoffman graph $\ho$. Then $S$ satisfies the following properties:
\begin{enumerate}

\item $S$ is a symmetric matrix with integral entries;

\item $S_{xx}\leq 0$;

\item $S_{xy}\leq 1$ if $x\neq y$;

\item $S_{xx}\leq S_{xy}$ for all $x, y$.

\end{enumerate}
\end{lemma}

The converse of this lemma is not true. For example, there is no Hoffman graph with special matrix $\begin{pmatrix}
-1 & -1 & -1\\
-1 & -1 & 1 \\
-1 & 1 & -1
\end{pmatrix}$.

\begin{lemma}(\cite[Corollary 3.3]{Woo})\label{interlacing}
If $\go$ is an induced Hoffman subgraph of a Hoffman graph $\ho$,
then $\lambda_{\min}(\go)\geq\lambda_{\min}(\ho)$ holds.
\end{lemma}

\subsection{Decompositions of Hoffman graphs}
Now, we introduce decompositions of Hoffman graphs and line Hoffman graphs.

\begin{definition}\label{directsummatrix}
Let $\{\ho^i\}_{i=1}^r$ be a family of Hoffman graphs. We say that $\{\ho^i\}_{i=1}^r$ is a \emph{decomposition} of a Hoffman graph $\ho$, if $\ho$ satisfies the following condition:

There exists a partition $\big\{V_{\rm slim}^1(\ho),V_{\rm slim}^2(\ho),\ldots,V_{\rm slim}^r(\ho) \big\}$ of $V_{\rm slim}(\ho)$ such that induced Hoffman subgraphs generated by $V_{\rm slim}^i(\ho)$ are $\ho^i$ for $i=1,2,\ldots,r$ and
$$S(\ho)=
\begin{pmatrix}
S(\ho^1) &&&\\
& S(\ho^2)&&\\
&&\ddots&\\
&&&S(\ho^r)
\end{pmatrix}
$$ is a block diagonal matrix with respect to this partition of $V_{\rm slim}(\ho)$.
\end{definition}

If a Hoffman graph $\ho$ has a decomposition $\{\ho^i\}_{i=1}^r$, then we write $\ho=\uplus_{i=1}^r\ho^i$.

A Hoffman graph $\ho$ is said to be \emph{decomposable}, if $\ho$ has a decomposition $\{\ho^i\}_{i=1}^r$ with $r\geq2$. Otherwise, $\ho$ is called \emph{indecomposable}. Note that $\ho$ is decomposable if and only if its special matrix $S(\ho)$ is a block diagonal matrix with at least $2$ blocks.

The following lemma gives a combinatorial way to define the decomposability of Hoffman graphs.

\begin{lemma}\label{combi}
Let $\ho$ be a Hoffman graph. A family $\{\ho^i\}_{i=1}^r$ of induced Hoffman subgraphs of $\ho$ is a decomposition of $\ho$, if it satisfies the following conditions:
\begin{enumerate}
\item $V(\ho)=\cup_{i=1}^rV(\ho^i)$;
\item $V_{\rm slim}(\ho^i)\cap V_{\rm slim}(\ho^j)=\emptyset$ for $i\neq j$;
\item if $x \in V_{\rm slim}(\ho^i),~f \in V_{\rm fat}(\ho)$, and $x\sim f$, then $f\in V_{\rm fat}(\ho^i);$
\item if $x \in V_{\rm slim}(\ho^i)$, $y \in V_{\rm slim}(\ho^j)$, and $i\neq j$, then $x$ and $y$ have at most one common fat neighbor, and they have one if and only if they are adjacent.
\end{enumerate}
\end{lemma}
\begin{proof} By direct verification. \end{proof}

%
%
%

\begin{definition}
Let $\Ho$ be a family of pairwise non-isomorphic Hoffman graphs. A Hoffman graph $\ho$ is called a \emph{$\Ho$-line Hoffman graph} if there exists a Hoffman graph $\ho^\prime$ satisfying the following conditions:
\begin{enumerate}
\item $\ho^\prime$ has $\ho$ as an induced Hoffman subgraph;
\item $\ho^\prime$ has the same slim graph as $\ho$;
\item $\ho^\prime=\uplus_{i=1}^{r}\ho_i^\prime$, where $\ho_i^\prime$ is isomorphic to an induced Hoffman subgraph of some Hoffman graph in $\Ho$ for $i=1,\dots,r$.
\end{enumerate}
\end{definition}

Let $t$ be a positive integer. The Hoffman graph with one slim vertex adjacent to $t$ fat vertices is the unique Hoffman graph with the special matrix $(-t)$. We denote it by $\mathfrak{h}^{(t)}$.

\begin{lemma}\label{t-cherry} Let $t$ be a positive integer. The $t$-fat Hoffman graph $\ho$ with special matrix $S(\ho)=J_n-(t+1)I_n$ is unique and this Hoffman graph is a $\big\{\ho^{(t+1)}\big\}$-line Hoffman graph.
\end{lemma}

\begin{proof}
Uniqueness is clear because the only possible Hoffman graph for $\ho$ is a Hoffman graph such that its slim
graph is $K_n$, and every slim vertex has its own $t$ fat neighbors. By adding one fat vertex which is adjacent to all slim vertices, we obtain a new Hoffman graph which can be decomposed as $n$ copies of $\ho^{(t+1)}$'s. Hence, $\ho$ is a $\big\{\ho^{(t+1)}\big\}$-line Hoffman graph.
\end{proof}

Hoffman and Ostrowski showed the following result. For a proof, see \cite[Theorem 2.14]{HJAT}.
\begin{theorem}\label{Ostrowski} Suppose $\ho$ is a Hoffman graph, and $f_1, \dots, f_r \in V_f(\ho)$. Let $\go^{n_1, \dots, n_r}$ be the Hoffman graph obtained from $\ho$ by replacing each $f_i$ by a slim $n_i$-clique $K^i$, and joining, for each $i$, every neighbor of $f_i$ with all vertices in $K^i$. Then
\begin{center}$\lambda_{\min}(\go^{n_1,\dots, n_r})\geq \lambda_{\min}(\ho)$,\end{center}
and
\begin{displaymath}\lim_{n_1,\dots,n_r \rightarrow \infty}\lambda_{\min}(\go^{n_1,\dots, n_r})= \lambda_{\min}(\ho).\end{displaymath}

\end{theorem}

\section{A structural theorem}
In this section, we give a structural theorem for $t$-fat Hoffman graphs with smallest eigenvalue at least $-t-1$. In Subsection \ref{Forbiddenmatrice}, we consider a set of matrices whose smallest eigenvalue is at most $-t-\sqrt{2}$. In Subsection \ref{familyofHoffmangraphs}, we introduce a finite family $\Go(t)$ of $t$-fat Hoffman graphs. In Corollary \ref{line1}, we will show that the $t$-fat Hoffman graphs with smallest eigenvalue greater than $-t-\sqrt{2}$ are $\Go(t)$-line Hoffman graphs.
\subsection{Forbidden matrices}\label{Forbiddenmatrice}

Two matrices $B_1$ and $B_2$ are \emph{equivalent} if there exists a permutation matrix $P$ such that $P^T B_1 P = B_2$. A symmetric matrix is said to be \emph{reducible} if and only if it is equivalent to a block diagonal matrix with at least two blocks. If a symmetric matrix is not reducible, it is said to be \emph{irreducible}.

\begin{definition}
Let $t$ and $a$ be two integers where $t>0$. We define the matrices $m_{1,a}(t)$, $m_{2,a}(t)$, $m_{3,a}(t)$, $m_{4,a}(t)$, $m_{5}(t)$, $m_{6}(t)$, $m_{7}(t)$, $m_{8}(t)$, $m_{9}(t)$ as follows:
\begin{enumerate}

\item $m_{1,a}(t)= \begin{pmatrix}
-t+a
\end{pmatrix}$;

\item $m_{2,a}(t)=\begin{pmatrix}
-t & a \\
a & -t
\end{pmatrix}$;

\item $m_{3,a}(t)=\begin{pmatrix}
-t-1 & a \\
a & -t
\end{pmatrix}$;

\item $m_{4,a}(t)=\begin{pmatrix}
-t-1 & a \\
a & -t-1
\end{pmatrix}$;

\item \begin{equation*}
\begin{array}[c]{ccc}

  m_{5}(t)=\begin{pmatrix}-t & -1 & -1\\-1 & -t & -1 \\-1 & -1 & -t\end{pmatrix},&
  m_{6}(t)=\begin{pmatrix}-t & 1 & 1\\1 & -t & -1 \\1 & -1 & -t\end{pmatrix},\\

  \\

  m_{7}(t)=\begin{pmatrix}-t & 0 & 1\\0 & -t & -1 \\1 & -1 & -t\end{pmatrix},&
  m_{8}(t)=\begin{pmatrix}-t & 0 & 1\\0 & -t & 1 \\1 & 1 & -t\end{pmatrix}, \\

  \\

  m_{9}(t)=\begin{pmatrix}-t & 0 & -1\\0 & -t & -1 \\-1 & -1 & -t\end{pmatrix}.\\
\end{array}
\end{equation*}
\end{enumerate}
\end{definition}

Note that the symmetric matrices $m_{2,0}(t),m_{3,0}(t)$ and $m_{4,0}(t)$ are reducible.

\begin{definition}\label{F(t)} Let $t$ be a positive integer. We define the set $\mathcal{M}(t)$ of irreducible symmetric matrices as the union of the sets $M_1(t),M_2(t)$, $M_3(t)$, $M_4(t)$ and $M_5(t)$, where
\begin{equation*}
\begin{split}
       M_1(t)& =\{m_{1,a}(t) \mid a=-2,-3,\ldots\},    \\
       M_2(t)& =\{m_{2,a}(t) \mid a=-2,-3,\ldots,-t\}, \\
       M_3(t)& =\{m_{3,a}(t) \mid a=1,-1,-2,\ldots, -t\}, \\
       M_4(t)& =\{m_{4,a}(t) \mid a=1,-1,-2,\ldots, -t-1\}, \\
       M_5(t)& =\{m_{5}(t),m_{6}(t),m_{7}(t),m_{8}(t),m_{9}(t)\}.
  \end{split}
\end{equation*}
\end{definition}


Note that $$\lambda_{\min}(m_{1,a}(t))=-t+a,~\lambda_{\min}(m_{2,a}(t))=-t-|a|, $$
$$\lambda_{\min}(m_{3,a}(t))=-t-\frac{1+\sqrt{1+a^2}}{2},$$
$$\lambda_{\min}(m_{4,a}(t))=-t-1-|a|,~\lambda_{\min}(m_{5}(t))=\lambda_{\min}(m_{6}(t))=-t-2,$$
$$\lambda_{\min}(m_{7}(t))=\lambda_{\min}(m_{8}(t))=\lambda_{\min}(m_{9}(t))=-t-\sqrt{2}.$$
It means that all matrices in $\mathcal{M}(t)$ have smallest eigenvalue at most $-t-\sqrt{2}$. Hence, we have the following result by Lemma \ref{matrixinterlace} (i).

\begin{proposition}\label{forbidden}
Let $t$ be a positive integer and $\ho$ be an indecomposable $t$-fat Hoffman graph. If the special matrix $S(\ho)$ contains a principal submatrix P with $\lambda_{\min}(P)\leq -t-\sqrt{2}$, then $\lambda_{\min}(\ho)\leq -t-\sqrt{2}$.
In particular, if $S(\ho)$ contains a principal submatrix equivalent to one of the matrices in $\mathcal{M}(t)$, then $\lambda_{\min}(\ho)\leq -t-\sqrt{2}$.
\end{proposition}

\subsection{A family of Hoffman graphs}\label{familyofHoffmangraphs}
In this subsection, we introduce the family $\Go(t)$ of $t$-fat Hoffman graphs.

\begin{theorem}\label{class}
Let $t$ be a positive integer. Suppose that $\ho$ is an indecomposable $t$-fat Hoffman graph such that none of the principal submatrices of $S(\ho)$ is equivalent to some element of $\mathcal{M}(t)$. Then $\lambda_{\min}(\ho)\geq -t-1$ and  $S(\ho)$ is one of the following:

\begin{enumerate}

\item $(-t)$, with smallest eigenvalue $-t$;

\item $(-t-1)$, with smallest eigenvalue $-t-1$;

\item $J-(t+1)I$, with smallest eigenvalue $-t-1$;

\item $\begin{pmatrix}
J_{r_1}-(t+1)I_{r_1} & -J_{r_1\times r_2} \\
-J_{r_2\times r_1}& J_{r_2}-(t+1)I_{r_2}
\end{pmatrix}$ for some positive integers $r_1$ and $r_2$, with smallest eigenvalue $-t-1$.

\end{enumerate}

Moreover, if $S(\ho)=\begin{pmatrix}
J_{r_1}-(t+1)I_{r_1} & -J_{r_1\times r_2} \\
-J_{r_2\times r_1}& J_{r_2}-(t+1)I_{r_2}
\end{pmatrix}$, then both $r_1$ and $r_2$ are at most $t$. In particular, $|V_{\rm slim}(\ho)|\leq 2t$.

\end{theorem}

\begin{proof}

Suppose that there exists a vertex $x \in V_{\rm slim}(\ho)$ such that $S(\ho)_{xx} \leq -t-1$. Since $\ho$ is indecomposable, we find that, by Lemma \ref{special} (iv) and Proposition \ref{forbidden}, $S(\ho)$ does not contain $m_{1,a}(t),~m_{3,b}(t)$ or $m_{4,b}(t)$ as a principal submatrix for $a=-2,-3,\ldots$ and $b=1,-1,-2,\ldots$. Therefore, we find that $S(\ho) = (-t-1)$ and hence $\ho=\ho^{(t+1)}$ with smallest eigenvalue $-t-1$.

Now, we may assume that $S(\ho)_{xx} = -t$ for all $x\in V_s(\ho)$. For distinct vertices $x$ and $y$, it is easy to see that  $S(\ho)_{xy} \in \{0,1,-1\}$ since $m_{2,a}(t)$ is not a principal submatrix of $S(\ho)$ for $a=-2,-3,\ldots$. Considering that $\ho$ is indecomposable and $S(\ho)$ does not contain any principal submatrix equivalent to $m_{7}(t),m_{8}(t)$ or $m_{9}(t)$, we obtain that $S(\ho)_{xy} \neq 0$. It follows that $S(\ho)_{xy} \in \{1,-1\}$.

We define a relation $\mathcal{R}$ on $V_{\rm slim}(\ho)$ by $x \mathcal{R} y$ if $S(\ho)_{xy} = 1$ or $x=y$ holds.

\vspace{0.2cm}

\noindent{\bf(Claim I)} The relation $\mathcal{R}$ on $V_{\rm slim}(\ho)$ is an equivalence relation.

\noindent\emph{(Proof of Claim I)} Clearly, $\mathcal{R}$ is reflexive and symmetric. For transitivity, suppose that $x\mathcal{R} y$ and $y\mathcal{R} z$ both hold. If $S(\ho)_{xz} = -1$, $S(\ho)$ contains a principal submatrix which is equivalent to $m_{6}(t)$. This gives a contradiction. Hence, $S(\ho)_{xz} = 1$ and $x \mathcal{R}z$ holds.

\vspace{0.1cm}

\noindent{\bf(Claim II)} The number of equivalence classes under $\mathcal{R}$ is at most $2$. If $V_s(\ho)$ has two equivalence classes, then each equivalence class has size at most $t$.

\noindent\emph{(Proof of Claim II)} Suppose that there are three equivalence classes. Take three vertices, one vertex from each class. Then, the principal submatrix indexed by these three vertices is the matrix $m_{5}(t)$, which is a contradiction. Hence, there are at most 2 equivalence classes under $\mathcal{R}$.

Suppose that $V_{\rm slim}(\ho)$ has two equivalence classes $C_1$ and $C_2$. Let $C_1=\{x_1, x_2,\dots,x_{r_1}\}$. Then for any vertex $y\in C_2$, the vertices $x_i$ and $y$ have at least one common fat neighbor. Note that $x_i$ and $x_j$ have no common fat neighbor if $i\neq j$. It follows that the size of $C_1$  is at most the number of fat neighbors of $y$, that is, $ r_1\leq t$. Similarly, we find that $r_2=|C_2|\leq t$.

\vspace{0.2cm}

If $V_{\rm slim}(\ho)$ has only one equivalence class under $\mathcal{R}$, then $S(\ho)=J-(t+1)I$ and we are in the case (i) or (iii) and $\lambda_{\min}(\ho)\geq -t-1$. If $V_{\rm slim}(\ho)$ has two equivalence classes under $\mathcal{R}$, then $S(\ho)$ is of the form

\[
S(\ho)=
\begin{pmatrix}
J_{r_1}-(t+1)I_{r_1} & -J_{r_1\times r_2} \\
-J_{r_2\times r_1}& J_{r_2}-(t+1)I_{r_2}
\end{pmatrix}
\]
in a labeling with respect to the two equivalence classes and $\lambda_{\min}(\ho)=-t-1$. In this case, each class has at most $t$ elements.
\end{proof}

If an indecomposable $t$-fat Hoffman graph $\ho$ satisfies the assumption of Theorem \ref{class}, then $\ho$ has smallest eigenvalue at least $-t-1$. Conversely, for an indecomposable  Hoffman graph $\ho$ with smallest eigenvalue at least $-t-1$, $S(\ho)$ does not contain a principal submatrix which is equivalent to an element of $\mathcal{M}(t)$. Hence, we obtain the corollary below.

\begin{corollary}\label{-t-1} Let $\ho$ be a $t$-fat Hoffman graph. Then the special matrix $S(\ho)$ does not contain any principal submatrix which is equivalent to an element of $\mathcal{M}(t)$ if and only if $\ho$ has smallest eigenvalue at least $-t-1$.
\end{corollary}


\begin{definition}\label{G(t)}
Let $t$ be a positive integer. We define $\Go(t)$ to be the family of pairwise non-isomorphic indecomposable $t$-fat Hoffman graphs whose special matrix is either $(-t-1)$ or $\begin{pmatrix}
J_{r_1}-(t+1)I_{r_1} & -J_{r_1\times r_2} \\
-J_{r_2\times r_1} & J_{r_2}-(t+1)I_{r_2}
\end{pmatrix}$, where $1\le r_1,r_2\le t$.
\end{definition}

Note that $\Go(t)$ is a finite family of Hoffman graphs and the Hoffman graph $\ho^{(t+1)}$ with the special matrix $(-t-1)$ belongs to $\Go(t)$.

Theorem \ref{class} shows the following result, which relates $\mathcal{M}(t)$ to $\Go(t)$.

\begin{theorem}\label{line} Let $t$ be a positive integer and $\ho$ be a $t$-fat Hoffman graph such that its special matrix $S(\ho)$ does not contain any principal submatrix equivalent to an element of $\mathcal{M}(t)$. Then $\ho$ is a $t$-fat $\Go(t)$-line Hoffman graph.

\end{theorem}

\begin{proof} Let $\{\ho_i\}_{i=1}^r$ be a decomposition of $\ho$, where $\ho_i$ is an indecomposable induced Hoffman subgraph of $\ho$ for all $i$. Note that $\ho = \uplus_{i=1}^r \ho_i$.
Without loss of generality, it suffices to show that $\ho_1$ is a $t$-fat $\Go(t)$-line Hoffman graph. Clearly, $\ho_1$ is $t$-fat and $S(\ho_1) $ does not contain a principal submatrix equivalent to an element of $\mathcal{M}(t)$. By Theorem \ref{class}, we have to consider the following three cases for $S(\ho_1)$:
\begin{itemize}
 \item[$(a)$]if $S(\ho_1)= (-t)$ or $(-t-1)$, then $\ho_1\in\big\{\ho^{(t)},~\ho^{(t+1)}\big\}$;

 \item[$(b)$]if $S(\ho_1)= J-(t+1)I$, then $\ho_1$ is a $\big\{\ho^{(t+1)}\big\}$-line Hoffman graph by Lemma \ref{t-cherry};

 \item[$(c)$]if $S(\ho_1)=\begin{pmatrix}
J_{r_1}-(t+1)I_{r_1} & -J_{r_1\times r_2} \\
-J_{r_2\times r_1}& J_{r_2}-(t+1)I_{r_2}
\end{pmatrix}$ for some positive integers $r_1$ and $r_2$, then $\ho_1 \in \Go(t)$.
\end{itemize}
Since $\ho^{(t+1)}\in\Go(t)$ and $\ho^{(t)}$ is an induced Hoffman subgraph of $\ho^{(t+1)}$, this finishes the proof.
\end{proof}

By Corollary \ref{-t-1} and Theorem \ref{line}, we have
\begin{corollary}\label{line1}
Let $t$ be a positive integer and $\ho$ be a $t$-fat Hoffman graph with smallest eigenvalue greater than $-t-\sqrt{2}$. Then $\ho$ is a $t$-fat $\Go(t)$-line Hoffman graph.
\end{corollary}
\section{Associated Hoffman graphs}

In this section, we first summarize some facts about associated Hoffman graphs and quasi-cliques as defined in \cite{HJJ}. We also will show that under some conditions, the associated Hoffman graph is a line Hoffman graph of a finite family of Hoffman graphs.

Let $m$ be a positive integer and let $G$ be a graph that does not contain $\widetilde{K}_{2m}$ as an induced subgraph, where $\widetilde{K}_{2m}$ is the graph on $2m+1$ vertices consisting of a complete graph $K_{2m}$ and a vertex $\infty$ which is adjacent to exactly $m$ vertices of the $K_{2m}$. Let $n\geq (m+1)^2$ be a positive integer. Let $\mathcal{C}(n)$ := $\{C\mid$ $C$ is a maximal clique of $G$ with at least $n$ vertices$\}$. Define the relation $\equiv_n^m$ on $\mathcal{C}(n)$ by $C_1\equiv_n^m C_2$ if each vertex $x\in C_1$ has at most $m-1$ non-neighbours in $C_2$ and each vertex $y\in C_2$ has at most $m-1$ non-neighbours in $C_1$. Then $\equiv_n^m$ is an equivalence relation as $n\geq (m+1)^2$ (see \cite[Lemma 3.1]{HJJ}).

Let $[C]_n^m$ denote the equivalence class of $\mathcal{C}(n)$ of $G$ under the equivalence relation $\equiv_n^m$ containing the maximal clique $C$ of $\mathcal{C}(n)$. We define the quasi-clique with respect to the pair $(m, n)$, $Q([C]_n^m)$, as the induced subgraph of $G$ on the set $\{x\in V(G)\mid$ $x$ has at most $m-1$ non-neighbours in $C\}$.

Let $[C_1]_n^m, \dots, [C_r]_n^m$ be equivalence classes of maximal cliques under $\equiv_n^m$. The \emph{associated Hoffman graph} $\go=\go(G,m,n)$ is the Hoffman graph satisfying the following conditions:
\begin{enumerate}
\item $V_{\rm slim}(\go) = V(G)$, $V_{\rm fat}(\go)=\{F_1,\dots,F_r\}$;
\item the slim graph of $\go$ equals $G$;
\item the fat vertex $F_i$ is adjacent precisely to all vertices in $Q([C_i]_n^m)$ for each $i$.
\end{enumerate}

Let $\ho$ be a Hoffman graph with $V_{\rm fat}(\ho)=\{F_1,\dots,F_r\}$ for some positive integer $r$.  The graph $G(\ho, n)$ is the slim graph of the Hoffman graph $\go^{n_1, \dots, n_r}$ (as defined in Theorem \ref{Ostrowski}) with $n_i=n$ for $i=1,\ldots,r$.

The following result is shown in \cite[Proposition 4.1]{HJJ}.

\begin{prop}\label{assohoff}
Let $G$ be a graph and let $m \geq 2, \phi \geq 1, \sigma \geq 1, p \geq 1$ be integers.
There exists a positive integer $n = n(m, \phi, \sigma,  p) \geq (m+1)^2$ such that for any integer $q \geq n$, and any Hoffman graph $\ho$ with at most $\phi$ fat vertices and at most $\sigma$ slim vertices, the graph  $G(\ho, p)$ is an induced subgraph of $G$, provided that the graph  $G$ satisfies the following conditions:
\begin{enumerate}
\item the  graph $G$ does not contain  $\widetilde{K}_{2m}$ as an induced subgraph,

\item its associated Hoffman graph $\go = \mathfrak{g}(G, m, q)$ contains $\ho$ as an induced Hoffman subgraph.

\end{enumerate}
\end{prop}

The theorem below is a modification of \cite[Theorem 6.2]{HJJ} and its proof follows the proof of \cite[Theorem 6.2]{HJJ}. We state it for the convenience of readers. It will be used in the next section.

\begin{theorem}\label{quasicliques}
Let $t$ be a positive integer and $m(t)=\min\{m \mid \lambda_{\min}(\widetilde{K}_{2m}) < -t-1\}$. Then there exists a positive integer $p(t)$ such that if $G$ is a graph with $\lambda_{\min}(G)\geq -t-1$, then for any integer $r\ge p(t)$, the quasi-cliques $Q_1, Q_2, \ldots, Q_c$ of $G$ with respect to the pair $(m(t), r)$ satisfy the following conditions:
\begin{enumerate}
\item the complement of $V(Q_i)$ has degree at most $t ^2$ for $i=1,2, \ldots, c$,
\item the intersection $V(Q_i)\cap V(Q_j)$ contains at most t vertices for $1 \leq i < j \leq c$.
\end{enumerate}
\end{theorem}

\begin{proof}
Let $H$ be a graph with $t^2 +2$ vertices such that at least one vertex has degree $t^2 +1$.
Then, by the Perron-Frobenius theorem (\cite{Ch&G}, Theorem 8.8.1 (b)), the largest eigenvalue of $H$ is at least $\sqrt{t^2 +1} > t$. Now, consider the fat Hoffman graph $\mathfrak{b}=\mathfrak{b}(H)$ with one fat vertex such that the slim graph of $\mathfrak{b}(H)$ is equal to the complement of $H$. Then $S(\mathfrak{b}) = -A({H}) -I.$ It follows that $\lambda_{\min}(\mathfrak{b}) \leq -\sqrt{t^2+1} -1 < -t-1.$
Hence, by Theorem \ref{Ostrowski} and Proposition \ref{assohoff}, there exists a positive integer $p_H\geq (m(t)+1)^2$ such that for any integer $r_1 \geq p_H$, the
 associated Hoffman graph $\go( G, m(t), r_1)$ does not contain
$\mathfrak{b}(H)$ as an induced Hoffman subgraph. Let $p_1$ be the maximum of all $p_H$, where the maximum is taken over all graphs $H$
having exactly $t^2 + 2$ vertices and containing a vertex with degree $t^2 + 1$.

Let $\mathfrak{y}$ be any fat Hoffman graph with exactly two fat vertices $F_1, F_2$ and $t+1$ slim vertices $x_1, x_2, \ldots, x_{t+1}$, such that each vertex $x_i$ is adjacent to $F_1$ and $F_2$, for $i=1, 2, \dots, t+1.$ Then the diagonal elements of $S(\mathfrak{y})$ are equal to $-2$ and the rest of the entries of $S(\mathfrak{y})$ are in $\{-2, -1\}$.
This shows that $\lambda_{\min}(\mathfrak{y}) \leq -t-2 < -t-1$.
Hence, by Theorem \ref{Ostrowski} and Proposition \ref{assohoff}, there exists a positive integer $p(\mathfrak{y}) \geq (m(t)+1)^2$ such that for any integer  $r_2 \geq p(\mathfrak{y})$, the  associated Hoffman graph $\go( G, m(t), r_2)$ does not contain
$\mathfrak{y}$ as an induced Hoffman subgraph. Let $p_2$ be the maximum of all such $p(\mathfrak{y})$'s.

Now take $p(t) = \max\{p_1, p_2\}$. Then the result follows.
\end{proof}

Now, by using Proposition \ref{assohoff}, we can show the theorem below.

\begin{theorem} \label{slimgraph}
Let $t$ be a positive integer and $m(t)=\min\{m \mid \lambda_{\min}(\widetilde{K}_{2m}) < -t-1\}$. Then there exists a positive integer $p^\prime(t)$ such that if G is a graph with $\lambda_{\min}(G)\geq -t-1$, then for any integer $r\ge p^\prime(t)$, the associated Hoffman graph $\go(G,m(t),r)$ does not contain an induced Hoffman subgraph whose special matrix is contained in $\mathcal{M}(t)$.

Moreover, if $\go(G,m(t),r)$ is $t$-fat for some $r\ge p^\prime(t)$, then $G$ is the slim graph of a $t$-fat $\Go(t)$-line Hoffman graph.
\end{theorem}
\begin{proof}
Let $G$ be a graph with $\lambda_{\min}(G)\geq -t-1$. First, it is clear that $G$ does not contain  $\widetilde{K}_{2m(t)}$ as an induced subgraph. We will show that there exists a positive integer $p^\prime(t)$ such that for $r\geq p^\prime(t)$, the associated Hoffman graph $\go(G,m(t),r)$ does not contain an induced Hoffman subgraph whose special matrix is an element of $\mathcal{M}(t)$. Note that every Hoffman graph whose special matrix is in $M_1(t)$ has only one slim vertex and contains $\ho^{(t+2)}$ as an induced Hoffman subgraph.
Let $\ho_1=\ho^{(t+2)}$ and $\ho_2, \dots, \ho_{t^\prime}$ be all of Hoffman graphs with at least 2 slim vertices whose special matrices are an element of $\mathcal{M}(t)$. Then each $\ho_i$ has at most 3 slim vertices and at most $3t+1$ fat vertices. Since $\lambda_{\min}(\ho_i)<-t-1$, there exists a positive integer $p_i$ such that $\lambda_{\min}(G(\ho_i, p_i))<-t-1$ for all $1\leq i \leq t^\prime$ by Theorem \ref{Ostrowski}. Let $p^\prime_i=n(m(t),3t+1,3,p_i)$ and $p^\prime(t) = \max_{1\leq i \leq t^\prime}p^\prime_i$. Then for $r\geq p^\prime(t)$, $\go(G,m(t),r)$ does not contain an induced Hoffman subgraph whose special matrix is an element of $\mathcal{M}(t)$ by Proposition \ref{assohoff}. If $\go(G,m(t),r)$ is $t$-fat for some $r\ge p^\prime(t)$, $\go(G, m(t), r)$ is a $t$-fat $\Go(t)$-line Hoffman graph by Theorem \ref{line}. Hence, $G$ is the slim graph of a $t$-fat $\Go(t)$-line Hoffman graph.
\end{proof}

\section{Main theorems}

Recall that the Ramsey number $R(a,b)$ is defined by the well-known Ramsey's Theorem.

\begin{theorem}\label{Ramsey}
(\cite{Ramsey}) Let a, b be two positive integers. Then there exists a minimum positive integer $R(a,b)$ such that for any graph $G$ on $n\geq R(a,b)$ vertices, the graph $G$ contains a clique of the order $a$ or an independent set of size $b$.
\end{theorem}

By using the concept of associated Hoffman graph and Ramsey theory, we will prove our main theorems as follows.

\begin{theorem}\label{main1}
Let $t \geq 2$ be an integer and $s\in\{t-1,t\}$. Then, there exists a positive integer $K(t)$, such that if a graph $G$ satisfies the following conditions:
\begin{enumerate}
\item $d(x)>K(t)$ holds for all $x \in V(G)$;
\item for all $x \in V(G)$, any $(t^2 +1)$-plex containing $x$ has the order at most $\frac{d(x)-K(t)}{s}$;
\item $\lambda_{\min} (G) \geq -t-1$,
\end{enumerate}
then the following hold:
\begin{itemize}
\item[$(a)$] If $s=t-1$, then $G$ is the slim graph of a $t$-fat $\Go(t)$-line Hoffman graph.
\item[$(b)$] If $s=t$, then $G$ is the slim graph of a $(t+1)$-fat $\ho^{(t+1)}$-line Hoffman graph.
\end{itemize}
\end{theorem}

\begin{proof}
Let $m(t)=\min\{m \mid \lambda_{\min}(\widetilde{K}_{2m}) < -t-1\}$. Then $G$ does not contain $\widetilde{K}_{2m(t)}$ as an induced subgraph. Choose $n(t)=\max\{p(t),p^\prime(t)\}$ where $p(t)$ and $p^\prime(t)$ are such that Theorems \ref{quasicliques} and \ref{slimgraph} both hold. Let $K(t)$ be the Ramsey number $R(n(t),(t+1)^2+1)$.

Pick a vertex $x$ and let $Q_1, \dots, Q_{c(x)}$ be the quasi-cliques with respect to the pair $(m(t),n(t))$ containing $x$. Suppose that $c(x)\leq s$. Then from Theorem \ref{quasicliques}
(i), we know that each $Q_i$ is a $(t^2+1)$-plex, so $|V(Q_i)|\leq \frac{d(x)-K(t)}{s}$ for $i=1,\ldots,c(x)$. We obtain
$$\big|\cup_{i=1}^{c(x)}V(Q_i)\big|\leq \sum\nolimits_{i=1}^{c(x)}|V(Q_i)|\leq \frac{(d(x)-K(t))c(x)}{s} \leq d(x)-K(t).$$
Now, $V(G_x)-\cup_i V(Q_i)$ has at least $K(t)$ vertices and $G$ does not contain $((t+1)^2+1)$-claw as an induced subgraph. Hence we can find a new clique with at least $n(t)$ vertices in $V(G_x) - \cup_i V(Q_i)$, and a new quasi-clique $Q_{c(x)+1}$ containing $x$. This makes a contradiction, so we obtain $c(x)\geq s+1$. This implies that each vertex $x$ in $G$ is contained in at least $s+1$ quasi-cliques with respect to the pair $(m(t),n(t))$ and the associated Hoffman graph $\go=\go(G,m(t),n(t))$ is $(s+1)$-fat.

If $s= t-1$, the graph $G$ is the slim graph of a $t$-fat $\Go(t)$-line Hoffman graph by Theorem \ref{slimgraph}. This shows (a).

If $s=t$, the graph $G$ is the slim graph of a $(t+1)$-fat $\Go(t)$-line Hoffman graph. However, the only $(t+1)$-fat Hoffman graph in $\Go(t)$ is $\ho^{(t+1)}$. Therefore $G$ is the slim graph of a $(t+1)$-fat $\{\ho^{(t+1)}\}$-line Hoffman graph and (b) follows.
\end{proof}

\vspace{0.2cm}
\begin{proof}[{\bf Proof of Theorem \ref{intro1}}] Let $K(t)$ be the smallest integer such that Theorem \ref{main1} holds. Then by using Theorem \ref{main1}, we find that $G$ is the slim graph of some Hoffman graph $\ho=\uplus_{i=1}^{v}\ho_i$ where $\ho_i$ is isomorphic to $\ho^{(t+1)}$ for $i=1,2,\ldots,v$.

Now define a hypergraph $H$ with vertex set $V_{\rm fat}(\ho)$ and (hyper)edge set $\big\{N_{\ho}^{\rm fat}(x) \mid x\in V_{\rm slim}(\ho)\big\}$. Note that $H$ is a linear $(t+1)$-uniform hypergraph. Let $x,y\in V_{\rm slim}(\ho)$, then $x$ and $y$ are adjacent in $G$ if and only if $|N_{\ho}^{\rm fat}(x)\cap N_{\ho}^{\rm fat}(y)|=1$. This shows that $G$ is the intersection graph of $H$.
\end{proof}

\vspace{0.2cm}
\begin{proof}[{\bf Proof of Theorem \ref{intro3}}]Let $K=K(2)$ be the smallest integer such that Theorem \ref{main1} holds. By using Theorem \ref{main1}, we find that $G$ is the slim graph of a $2$-fat $\Go(2)$-line Hoffman graph. Now we determine the finite family $\Go(2)$ of Hoffman graphs. For a Hoffman graph $\ho$ in $\Go(2)$,
\begin{itemize}
\item[$1.$] if the special matrix of $\ho$ is equal to $(-3)$, then $\ho$ is the Hoffman graph $\ho^{(3)}$, that is $\mathfrak{f}_1$;
\item[$2.$] if the special matrix of $\ho$ is $\left(
                                               \begin{array}{cc}
                                                 J_{r_1}-3I_{r_1} & -J_{r_1\times r_2} \\
                                                 -J_{r_2\times r_1} & J_{r_2}-3I_{r_2} \\
                                               \end{array}
                                             \right)$, where $r_1,r_2$ are integers such that $1\leq r_1,r_2\le 2$, then $\ho$ will be given as follows:

$(a)$ when $r_1=r_2=2$, then $\ho$ is the Hoffman graph $\mathfrak{f}_2$;

$(b)$ when $r_1=2,r_2=1$ or $r_1=1,r_2=2$, then $\ho$ is the Hoffman graph $\mathfrak{f}_{2,1}$ in Figure \ref{fsubf};

$(c)$ when $r_1=r_2=1$, then $\ho$ is the Hoffman graph $\mathfrak{f}_{2,2}$ in Figure \ref{fsubf} or $\mathfrak{f}_3$.

\begin{center}
\begin{figure}[H]
\includegraphics[scale=1.0]{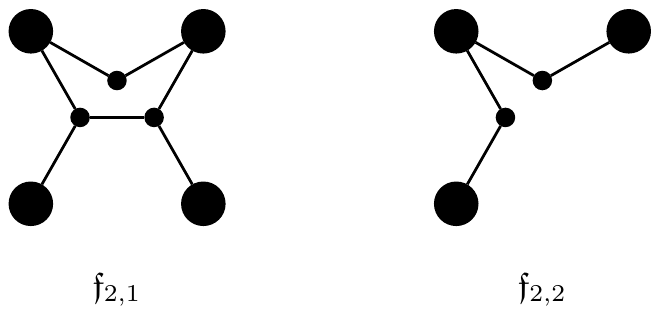}
{\caption{}\label{fsubf}}
\end{figure}
\end{center}
\vspace{-1cm}
\end{itemize}
Thus, $\mathfrak{G}(2)=\{\mathfrak{f}_1,\mathfrak{f}_2,\mathfrak{f}_{2,1},\mathfrak{f}_{2,2},\mathfrak{f}_3\}$ and $G$ is the slim graph of a $2$-fat $\{\mathfrak{f}_1,\mathfrak{f}_2,\mathfrak{f}_{2,1},\mathfrak{f}_{2,2},\mathfrak{f}_3\}$-line Hoffman graph.

Since $\mathfrak{f}_{2,1}$ and $\mathfrak{f}_{2,2}$ are induced Hoffman subgraphs of $\mathfrak{f}_2$, it is easy to check that $G$ is also the slim graph of a $2$-fat $\{\mathfrak{f}_1,\mathfrak{f}_2,\mathfrak{f}_3\}$-line Hoffman graph.
\end{proof}

\begin{theorem}\label{main2}
Let $t \geq 2$ be an integer and $s\in\{t-1,~t\}$. Then, there exists a positive integer $\kappa(t)$, such that if a graph $G$ satisfies the following conditions:
\begin{enumerate}
\item $d(x)>\kappa(t)$ holds for all $x \in V(G)$;

\item for all $x \in V(G)$, $\bar{d}(G_x) \leq \frac{d(x)-\kappa(t)}{s}$ holds;

\item $\lambda_{\min} (G) \geq -t-1$,
\end{enumerate}
then the following hold:
\begin{itemize}
\item [$(a)$]If $s=t-1$, then $G$ is the slim graph of a $t$-fat $\Go(t)$-line Hoffman graph.
\item [$(b)$]If $s=t$, then $G$ is the slim graph of a $(t+1)$-fat $\ho^{(t+1)}$-line Hoffman graph.
\end{itemize}
\end{theorem}

\begin{proof}
Let $m(t)=\min\{m \mid \lambda_{\min}(\widetilde{K}_{2m}) < -t-1\}$, then $G$ does not contain $\widetilde{K}_{2m(t)}$ as an induced subgraph. Choose $n(t)=\max\{p(t),p^\prime(t),2t^2-t+2\}$ where $p(t)$ and $p^\prime(t)$ are such that Theorems \ref{quasicliques} and \ref{slimgraph} both hold. Let $\kappa(t) :=2R(n(t),(t+1)^2+1)+(t^2+1)t$ be an integer. It suffices to show that the associated Hoffman graph $\go(G,m(t),n(t))$ is $(s+1)$-fat.

Pick a vertex $x$ and let $Q_1, \dots, Q_{c(x)}$ be the quasi-cliques with respect to the pair $(m(t),n(t))$ containing $x$. Suppose that $c(x)\leq s$. Let  $Q'_i$ be a subgraph of $G$ induced by $V(G_x) \cap V(Q_i)-\cup_{j<i}V(Q_j)$ and let $\alpha_i=|V(Q'_i)|$ for $i=1,\ldots,c(x)$. By Theorem \ref{quasicliques} (ii), we know that $|V(Q'_i)|\geq |V(Q_i)|-(i-1)t\geq n(t)-(c(x)-1)t\geq n(t)-(s-1)t>t^2+1$. Note that $\sum\nolimits_{i=1}^{c(x)}\alpha_i=d(x)-d'$ with $d'<R(n(t),(t+1)^2+1)$.

Then by using Theorem \ref{quasicliques} (i)
\begin{equation*}
\begin{split}
   2|E(G_x)|&\geq 2\sum\nolimits_{i=1}^{c(x)}|E(Q'_i)|\\
                  &\geq\sum\nolimits_{i=1}^{c(x)}\alpha_i(\alpha_i-1-t^2)\\
                  &=\sum\nolimits_{i=1}^{c(x)}\alpha_i^2-(t^2+1)\sum\nolimits_{i=1}^{c(x)}\alpha_i\\
                  &\geq \frac{(\sum\nolimits_{i=1}^{c(x)}\alpha_i)^2}{c(x)}-(t^2+1)(d(x)-d')\\
                  &\geq\frac{(d(x)-d')^2}{s}-(t^2+1)(d(x)-d').
\end{split}
\end{equation*}

By (ii), $2|E(G_x)| = \bar{d}(G_x) d(x) \leq \frac{d(x)-\kappa(t)}{s}d(x)$. Now we have
\begin{equation*}\begin{split}
0&\geq\frac{(d(x)-d')^2}{s}-(t^2+1)(d(x)-d')-\frac{(d(x)-\kappa(t))d(x)}{s}\\
&=\frac{\kappa(t)-2d'-(t^2+1)s}{s}d(x)+\frac{d'^2}{s}+(t^2+1)d'.\\
\end{split}
\end{equation*}
Since $\kappa(t)=2R(n(t),(t+1)^2+1)+(t^2+1)t>2d'+(t^2+1)s$, it gives a contradiction.

Hence, we can conclude that the associated Hoffman graph $\go(G, m(t), n(t))$ is $(s+1)$-fat. This finishes the proof.
\end{proof}

\vspace{0.2cm}
\begin{proof}[{\bf Proof of Theorem \ref{intro2}}]Let $\kappa(t)$ be the smallest integer such that Theorem \ref{main2} holds. Then by Theorem \ref{main2}, we know that $G$ is the slim graph of a $(t+1)$-fat $\{\ho^{(t+1)}\}$-line Hoffman graph. As we did in the proof of Theorem \ref{intro1}, it is easy to see that $G$ is the intersection graph of some linear $(t+1)$-uniform hypergraph.
\end{proof}

\vspace{0.2cm}
\begin{proof}[{\bf Proof of Theorem \ref{intro4}}]Let $\kappa=\kappa(2)$ be the smallest integer such that Theorem \ref{main2} holds. Then by Theorem \ref{main2}, we know that $G$ is the slim graph of a $2$-fat $\Go(2)$-line Hoffman graph. Similarly, by the proof of Theorem \ref{intro3}, $G$ is the slim graph of a $2$-fat $\{\mathfrak{f}_1,\mathfrak{f}_2,\mathfrak{f}_3\}$-line Hoffman graph.
\end{proof}

\section{Applications}
\subsection{The graphs cospectral with $H(3,q)$}
Let $D\geq1$ and $q\ge 2$ be integers. Let $X$ be a finite set of size $q$. The Hamming graph $H(D,q)$ is the graph with the vertex set $X^D:=\prod_{i=1}^DX$ (the Cartesian product of $D$ copies of $X$) where two vertices are adjacent whenever they differ in precisely one coordinate. The graph $H(D,q)$ has distinct eigenvalues $\lambda_i=q(D-i)-D$ with multiplicities $m_i=\binom{D}{i}(q-1)^i$ for $i=0,1,\ldots,D$. In particular, the Hamming graph $H(3,q)$ has spectrum $\big\{(3q-3)^1,(2q-3)^{3(q-1)},(q-3)^{3(q-1)^2},(-3)^{(q-1)^3}\big\}$. The following result was first shown by S. Bang et al. \cite{Bang}.

\begin{theorem}There exists a positive integer $q^\prime$ such that for every integer $q \geq q^\prime$, any graph that is cospectral with the Hamming graph $H(3,q)$ is the intersection graph of some linear $3$-uniform hypergraph.
\end{theorem}
\begin{proof}
Let $G$ be a graph cospectral with $H(3,q)$ and let $A$ be the adjacency matrix of $G$. 
By (\ref{averagevalency}), we know that for any vertex $x$ in $G$, $$\bar{d}(G_x)=q-2.$$
Let $q^\prime=\kappa(2)-1$, where $\kappa(2)$ is an integer such that Theorem \ref{intro2} holds. Then by Theorem \ref{intro2}, the result follows.
\end{proof}

\subsection{The graphs cospectral with $J(v,3)$}
Let $v,~p\ge 2$ be integers. Let $X$ be a finite set of size $v$. The Johnson graph $J(v,p)$ is the graph with vertex set $\binom{X}{p}$, the set of all $p$-subsets of $X$, where two $p$-subsets are adjacent if they intersect in precisely $p-1$ elements. The graph $J(v,p)$ has distinct eigenvalues $\lambda_i=(p-i)(v-p-i)-i$ with multiplicities $m_i=\binom{v}{i}-\binom{v}{i-1}$ for $i=0,1,\ldots,\min\{v-p,p\}$. (Since the Johnson graph $J(v,p)$ is isomorphic to the Johnson graph $J(v,v-p)$, we always assume that $v\ge 2p$.) In particular, the Johnson graph $J(v,3)$ has spectrum $\big\{(3v-9)^1,(2v-9)^{v-1},(v-7)^{\frac{v(v-3)}{2}},(-3)^{\frac{v(v-1)(v-5)}{6}}\big\}$. Now we show the following result.
\begin{theorem}\label{j(v,3)}
There exists a positive integer $v^\prime$ such that for every integer $v \geq v^\prime$, any graph that is cospectral with the Johnson graph $J(v,3)$ is the intersection graph of some linear $3$-uniform hypergraph.
\end{theorem}
\begin{proof}
Let $G$ be a graph cospectral with $J(v,3)$ and let $A$ be the adjacency matrix of $G$. 
By (\ref{averagevalency}), we know that for any vertex $x$ in $G$, $$\bar{d}(G_x)=v-2.$$

Let $v^\prime=\kappa(2)+5$, where $\kappa(2)$ is an integer such that Theorem \ref{intro2} holds. Then by Theorem \ref{intro2}, the result follows.
\end{proof}

\subsection{The graphs cospectral with the $2$-clique extension of $(t_1\times t_2)$-grid}
The graph $(t_1\times t_2)$-grid is the line graph of the complete bipartite graph, $K_{t_1,t_2}$. In other words, it is the graph $K_{t_1}\Box K_{t_2}$, where $\Box$ represents the \emph{Cartesian product}. The spectrum of the $(t_1\times t_2)$-grid is $\big\{(t_1+t_2-2)^1,(t_1-2)^{t_2-1},(t_2-2)^{t_1-1},(-2)^{(t_1-1)(t_2-1)}\big\}$ (see, for example, Section $9.7$ of \cite{Ch&G}).

A graph $\widetilde{G}$ is called the $2$-clique extension of graph $G$ if $\widetilde{G}$ has matrix $\left(
                                                                                                          \begin{array}{cc}
                                                                                                            A-I & A \\
                                                                                                            A & A-I \\
                                                                                                          \end{array}
                                                                                                        \right)$ as its adjacency matrix, where $A$ is the adjacency matrix of $G$. So, if
                                                                                                        $G$ has spectrum
 \begin{equation*}
 \big\{\lambda_0^{m_0},\lambda_1^{m_1},\ldots,\lambda_t^{m_t}\big\},
 \end{equation*}
then the spectrum of $\widetilde{G}$ is
 \begin{equation*}
 \big\{(2\lambda_0+1)^{m_0},(2\lambda_1+1)^{m_1},\ldots,(2\lambda_t+1)^{m_t},(-1)^{(m_0+m_1+\cdots+m_t)}\big\}.
 \end{equation*}
Now we show the following result.
\begin{theorem}There exists a positive integer $t$ such that for all integers $t_1 \geq t_2 \geq t$, any graph that is cospectral with the $2$-clique extension of $(t_1 \times t_2)$-grid is the slim graph of a $2$-fat $\{\mathfrak{f}_1,\mathfrak{f}_2,\mathfrak{f}_3\}$-line Hoffman graph.
\end{theorem}
\begin{proof}Let $G$ be a graph cospectral with the $2$-clique extension of $(t_1\times t_2)$-grid with $t_1 \geq t_2$, then the spectrum of $G$ is
 $$\big\{(2t_1+2t_2-3)^1,(2t_1-3)^{t_2-1},(2t_2-3)^{t_1-1},(-1)^{t_1t_2},(-3)^{(t_1-1)(t_2-1)}\big\}.$$

For any vertex $x$ in $G$, let $P$ be a $5$-plex containing $x$. Then by Lemma \ref{plexorder}, we know that $$|V(P)|\leq\frac{2t_1t_2(5+2t_1-3)}{2t_1t_2-(2t_1+2t_2-3)+(2t_1-3)}=\frac{2t_1(t_1+1)}{t_1-1}.$$

Since $d(x)-\frac{2t_1(t_1+1)}{t_1-1}=\frac{2t_1t_2-7t_1-2t_2+3}{t_1-1}\geq \frac{2t_1t_2-9t_1+3}{t_1}\geq 2t_2-9.$

So let $t=\lceil\frac{K+9}{2}\rceil$, where $K$ is an integer such that Theorem \ref{intro3} holds. Then, by using Theorem \ref{intro3}, the result follows.

\end{proof}

\noindent

{\bf Acknowledgments} J.H. Koolen is partially supported by the National Natural Science Foundation of China (No. 11471009 and No. 11671376). Part of this work was done while J.Y. Yang was a graduate student in POSTECH.


\end{document}